\documentclass[reqno]{amsart}

\pdfoutput=1

\usepackage{amssymb}
\usepackage{booktabs}
\usepackage[british]{babel}
\usepackage{color}
\usepackage{enumitem}
\usepackage[T1]{fontenc}
\usepackage[utf8]{inputenc}
\usepackage{mathrsfs}
\usepackage[l2tabu,orthodox]{nag}
\usepackage{rotating}
\usepackage{tensor}
\usepackage{tikz-cd}

\usepackage[
		pdftex,
		bookmarks,
		bookmarksopen,
		bookmarksopenlevel=1,
		bookmarksnumbered,
		breaklinks,
		colorlinks,
		linkcolor=linkcolor,
		citecolor=citecolor,
		urlcolor=urlcolor,
	]{hyperref}

\hypersetup{
	pdftitle    = {An algebraic geometric classification of superintegrable systems in the Euclidean plane},
	pdfauthor   = {Jonathan Kress \& Konrad Sch\"obel},
	pdfsubject  = {article},
	pdfcreator  = {\LaTeX{} with `amsart' class},
	pdfproducer = {pdf\LaTeX},
	pdfkeywords = {superintegrable systems}
}

\definecolor{linkcolor}{rgb}{0.5,0.0,0.0}
\definecolor{citecolor}{rgb}{0.0,0.5,0.0}
\definecolor{urlcolor} {rgb}{0.0,0.0,0.5}

\title[Superintegrable systems in the Euclidean plane]
	{An algebraic geometric classification of superintegrable systems in the Euclidean plane}

\author{Jonathan Kress}
\email{j.kress@unsw.edu.au}
\address{
	School of Mathematics and Statistics \\
	The University of New South Wales \\
	Sydney 2052 \\
	Australia
}
\author{Konrad Schöbel}
\email{konrad.schoebel@uni-jena.de}
\address{
	Institut für Mathematik \\
	Fakultät für Mathematik und Informatik \\
	Friedrich-Schiller-Universität Jena \\
	07737 Jena \\
	Germany
}

\numberwithin{equation}{section}

\newtheorem{theorem}{Theorem}[section]
\newtheorem{proposition}[theorem]{Proposition}
\newtheorem{lemma}[theorem]{Lemma}
\newtheorem{corollary}[theorem]{Corollary}
\newtheorem*{conjecture}{Conjecture}
\theoremstyle{definition}
\newtheorem{definition}[theorem]{Definition}
\theoremstyle{remark}
\newtheorem{remark}[theorem]{Remark}

\setlist[enumerate,1]{label=(\roman*)}

\newcommand{\R}{\mathbb R}
\newcommand{\C}{\mathbb C}
\renewcommand{\P}{\mathbb P}

\renewcommand{\le}{\leqslant}
\newcommand{\coloneq}{:=}
\newcommand{\eqcolon}{=:}

\DeclareMathOperator{\tr}{tr}

\begin{document}

\begin{abstract}
	We prove that the set of non-degenerate second order maximally
	superintegrable systems in the complex Euclidean plane carries a natural
	structure of a projective variety, equipped with a linear isometry group
	action.  This is done by deriving the corresponding system of homogeneous
	algebraic equations.  We then solve these equations explicitly and give a
	detailed analysis of the algebraic geometric structure of the
	corresponding projective variety.  This naturally associates a unique
	planar line triple arrangement to every superintegrable system, providing
	a geometric realisation of this variety and an intrinsic labelling scheme.
	In particular, our results confirm the known classification by
	independent, purely algebraic means.
\end{abstract}

\maketitle

\tableofcontents

%============================================================================%
\section{Introduction}

\subsection{Superintegrable systems}

In classical and quantum mechanics exact solutions to the equations of motion play a central role, providing models with which to explore the properties of such systems and as a basis for perturbations.  Two systems in particular stand out, the harmonic oscillator, which represents the lowest order term in the Taylor expansion of any non-singular potential, and the Kepler-Coulomb potential of the classical celestial motion and the quantum Hydrogen atom.

Finding exact solutions to PDEs generally requires the use of symmetry methods.  In classical mechanics, continuous symmetries describe flows in the phase space that provide conserved quantities.  In quantum mechanics, differential operators commuting with the Hamiltonian preserve its energy eigenspaces.  Superintegrable systems are those systems with the greatest number of independent symmetries and the harmonic oscillator and Kepler-Coulomb system are the best known examples.  For a natural Hamiltonian system on an $n$-dimensional manifold, this maximum number is $2n-1$.

In the classical case, superintegrability confines the trajectories of the system to the level sets of $2n-1$ conserved quantities in the $2n$-dimensional phase space and hence to one-dimensional orbits which necessarily close if they are confined.  In the quantum case, maximal superintegrability is associated with quasi-exact solvability and in many cases, the energy spectrum has been determined from the symmetry algebra alone.

Not only are the harmonic oscillator and Kepler-Coulomb systems superintegrable, but their superintegrability is due to symmetries that are quadratic in the momenta, or second order differential operators in the quantum case.  Second order symmetries allow the equations of motion to be solved by separation of variables and (non-degenerate) second order superintegrable systems in two and three dimensions allow separation in more than one distinct system of coordinates \cite{Kalnins&Kress&Pogosyan&Miller,Kalnins&Kress&MillerIV}.

The separated solutions of the harmonic oscillator and Kepler-Coulomb system are given in terms of an orthonormal basis of special functions.  When separation occurs in more than one system, the interbasis expansion coefficients are also given in terms of special functions in the form of orthogonal polynomials.  These connections with orthogonal polynomials have been exploited by Post {\em et al}\ to ``explain'' the Askey-Wilson scheme of hypergeometric orthogonal polynomials with Wigner-İnönü like contractions between second order superintegrable systems providing the degenerations between the classes of orthogonal polynomials \cite{Kalnins&Miller&Post}.  It is this observation that motivates the current work that seeks to describe second order superintegrable systems as a projective variety.

Recently, there has been rapid growth in the number of known families of superintegrable systems, in particular those with higher order symmetries.  For an overview see the topical review of Miller, Post and Winternitz \cite{Miller&Post&Winternitz}.  However, many questions remain open and even the complete classification for non-degenerate second order superintegrability in constant curvature and conformally flat spaces is only known for two and three dimensions \cite{Kalnins&Kress&MillerI,Capel&Kress}.  This paper aims to develop new techniques that will provide a classification in any dimension.

%----------------------------------------------------------------------------%
\subsection{Motivation}

Our approach is strongly motivated by recent results on separable systems.
While the classification of separable systems on constant scalar curvature
manifolds has been known for over 30~years \cite{Kalnins&Miller,Kalnins}, a
recent algebraic geometric approach developed by the second
author~\cite{Schoebel12,Schoebel14,Schoebel15,Schoebel16} revealed a deep
algebraic and geometric structure underlying this classification.  Together
with A.~P.~Veselov he proved that the space of all separable systems on an
$n$-dimensional sphere (in normal form) carries a natural structure of a
projective variety, isomorphic to the real part~$\bar{\mathscr M}_{0,n+2}(\R)$
of the \emph{Deligne-Mumford moduli space}~$\bar{\mathscr M}_{0,n+2}$ of
stable algebraic curves of genus zero with~$n+2$ marked
points~\cite{Schoebel&Veselov}.  Exploiting the structure of these moduli
spaces led to a topological classification of separable systems on spheres by
\emph{Stasheff polytopes} and to a simple explicit construction based on a
natural \emph{operad structure} on the sequence of moduli
spaces~$\bar{\mathscr M}_{0,n}(\R)$.

We regard the present work as a proof of concept that the same ideas apply to
the classification of superintegrable systems, initiating an
``algebro-geometrisation'' of the classification of superintegrable systems
and its applications.  Indeed, leading experts in superintegrability consider
algebraic geometric methods the most promising route to mayor advances in this
field \cite{Miller&Post&Winternitz}:
\begin{quote}
	``The possibility of using methods of algebraic geometry to classify
	superintegrable systems is very promising and suggests a method to extend
	the analysis in arbitrary dimension as well as a way to understand the
	geometry underpinning superintegrable systems.''
\end{quote}
Even though, a concrete and detailed concept how to achieve this goal has
never been proposed.  This gap is filled with our proof of concept, which can
be better extended to higher dimensions and more general context than the
classical approach.

In the case of separable systems, a thorough analysis of the least non-trivial
example, the $3$-sphere, provided enough information for a generalisation to
arbitrary dimensions.  Since separable systems are closely related to
superintegrable systems, this suggests that an algebraic geometric description
of superintegrable systems in low dimensions will reveal sufficient additional
structure to push the classification further to higher dimensions, a task
which currently seems intractable by standard methods.  The present article is
a first step in this direction.

%----------------------------------------------------------------------------%
\subsection{Results}

Traditionally, classifying superintegrable systems of a certain type always
meant to give lists of normal forms for the equivalence classes of such
systems under isometries.  This sense of the word ``classification'' ignores
the fact that the set of superintegrable systems has a topological structure
and possibly an even more fine-grained geometric structure.  The first main
result we prove in this article is that the set of non-degenerate second order
maximally superintegrable systems in the Euclidean plane has the structure of
(a linear bundle over) a projective variety, equipped with a linear isometry
group action.  We call this variety the \emph{variety of superintegrable
systems}.

In other words, the natural category in which to consider the classification
problem for superintegrable systems is the category of projective varieties
equipped with linear group actions.  This indicates that the classification
problem can be better treated by algebraic geometric means, studying the
underlying algebraic equations instead of partial differential equations.  The
remaining part of this article is dedicated to substantiate this claim by
showing that a consequent algebraic geometric treatment not only simplifies
the classification problem considerably, but also reveals a deep and
previously hidden geometric structure.  We show, for instance, that
non-degenerate superintegrable systems on the plane are (essentially)
parametrised by a completely reducible ternary cubic.  This associates a
planar arrangement of projective line triples to each superintegrable system
in the plane, which supports the singular locus of the corresponding
superintegrable potential.

In this sense our work not only confirms the known classification by
independent means, but also enriches it with additional algebraic and
geometric structure.

%----------------------------------------------------------------------------%
\subsection{Comparison to prior work}

It should be noted that the idea to use varieties in order to classify
superintegrable systems is not new and appeared in earlier works of the first
author, along with Kalnins {\em et al}
\cite{Kalnins&Kress&Miller,Kalnins&Kress&Miller3D} and Capel
\cite{Capel&Kress}, where inhomogeneous polynomial integrability conditions
were used to classify superintegrable systems in two and three dimensions.
Those varieties depend on a non-canonical choice of some generic base point in
the manifold and carry an intricate non-linear isometry group action.  This
makes the use of computer algebra inevitable and is the principal obstruction
for a generalisation to higher dimensions.  We remark that the algebraic
geometric structure of these varieties has never been exploited for the
classification.

While we make use of the same polynomial integrability conditions, our
approach is fundamentally different in that here we use algebraic geometry
right from the beginning to \emph{define the setting} and not merely as a
\emph{tool to analyse} these conditions.  This has a number of advantages.
First, our variety does not depend on any additional choices and it carries a
linear isometry group action.  Moreover, although the polynomial integrability
conditions are inhomogeneous, it turns out to be a projective variety.
Second, our variety is simple enough to be analysed without the help of
computer algebra.  Third, the line triple arrangements and the multiplicities
of the ternary cubic mentioned above provide intrinsic geometric and algebraic
labelling schemes for superintegrable systems respectively their isometry
classes, in contrast to the known classification scheme where isometry classes
were labelled arbitrarily.  Note that the more structure on the set of
superintegrable systems that is revealed, the more likely it is to find patterns that
generalise to higher dimensions.

%----------------------------------------------------------------------------%
\subsection{Prospects}

It has been observed recently that the families of hypergeometric orthogonal
polynomials in the Askey scheme can be associated to second order
superintegrable systems on spaces of constant curvature, such that limiting
cases of these polynomials can be derived from contractions of superintegrable
systems \cite{Kalnins&Miller&Post}.  On the other hand, an extension of the
methods and results in the present article will eventually reveal a natural
structure of a projective variety on the set of second order superintegrable
systems on spaces of constant curvature.  Combined with the above mentioned
observation, we arrive at the following conjecture.

\begin{conjecture}
	The set of hypergeometric orthogonal polynomials in the Askey scheme
	carries a natural structure of a projective variety, equipped with a
	linear action of the general linear group $\mathrm{GL}(3)$.  Moreover, the
	graph of orbits under this action and their degenerations reproduces the
	Askey scheme as well as the classification of second order superintegrable
	systems in dimension two.
\end{conjecture}

This conjecture is supported by the fact that the Askey scheme carries the
structure of three manifolds with corners glued together \cite{Koornwinder}.
As of yet, there is no generalisation of the Askey scheme to higher dimensions
or more general special functions known, but superintegrable systems seem to
provide the right context for that.

A parametrisation via a projective variety would allow to study hypergeometric
orthogonal polynomials using a broad range of algebraic geometric methods and
on a global level.  What we propose here is a paradigm shift.  While
hypergeometric polynomials or, more generally, special functions have always
been regarded as many different families, each of them of the form
$P_\alpha(x)$ with the parameter $\alpha$ varying in some open subset of
$\R^n$, our results suggest to consider them as a single family $P_\alpha(x)$,
where $\alpha$ now varies in an algebraic variety.

It is expected that a two-variable generalisation of the Askey scheme will be 
revealed in the structure of second order superintegrable systems on constant 
curvature spaces in three dimensions \cite{PostPrivate}.  Current work in 
progress extending the results reported here would then suggest a natural 
generalisation of the above conjecture to a multivariable Askey scheme.

%----------------------------------------------------------------------------%
\subsection{Method}

As will be explained in Section~\ref{sec:preliminaries}, a second order
maximally superintegrable system on an $n$-dimensional Riemannian manifold is
given by a (non-trivial) solution to the $2n-1$ equations
\begin{equation}
	\label{eq:KidV=dVi}
	dV^{(\alpha)}=K^{(\alpha)}dV,
	\qquad
	\alpha=0,1,\ldots,2n-2,
\end{equation}
for linearly independent second order Killing tensors $K^{(\alpha)}$ and
arbitrary potential functions $V^{(\alpha)}$, where $V=V^{(0)}$ is the
potential defining the Hamiltonian of the system.  The point of departure for
our approach to a classification of superintegrable systems is the observation
that this system, which at first glance is a non-linear system of partial
differential equations, is actually equivalent to a system of \emph{algebraic}
equations on a finite dimensional vector space.  This can be seen as follows.

First note that by definition the $K^{(\alpha)}$ belong to the finite
dimensional vector space of Killing tensors and are therefore, essentially,
algebraic objects -- as opposed to the arbitrary functions $V^{(\alpha)}$.  We
will therefore eliminate first the $V^{(\alpha)}$ for $\alpha\not=0$ and then
$V=V^{(0)}$ from the above equations, leaving equations on the $K^{(\alpha)}$
alone.

If $K^{(\alpha)}$ and $V$ are given, then the Equation~\eqref{eq:KidV=dVi} can
be used to obtain $V^{(\alpha)}$, provided the integrability condition
\begin{equation}
	\label{eq:dKidV=0}
	d(K^{(\alpha)}dV)=0
\end{equation}
is met.  This is the so-called \emph{Bertrand-Darboux condition} and already
eliminates the unknown functions $V^{(\alpha)}$ for $\alpha\not=0$.

If only the $K^{(\alpha)}$ are given, the Equations~\eqref{eq:dKidV=0} are
second order linear differential equations for $V$, the coefficients being
linear in the $K^{(\alpha)}$ and their first derivatives.  Following
\cite{Kalnins&Kress&Miller} we can use them to express all but one of $V$'s
second derivatives as linear combinations of $V$'s first derivatives, where
the coefficients are rational in the $K^{(\alpha)}$ and their first
derivatives.  In the case of the Euclidean plane considered here, for example,
this reads (in complex coordinates)
\footnote{%
	The factor $3/2$ chosen here for convenience differs from the convention
	used in \cite{Kalnins&Kress&Miller}.
}
\begin{equation}
	\label{eq:V''=CV'}
	\begin{bmatrix}
		V_{zz}\\
		V_{ww}
	\end{bmatrix}
	=
	\frac32
	\begin{bmatrix}
		C_{11}&C_{12}\\
		C_{21}&C_{22}
	\end{bmatrix}
	\begin{bmatrix}
		V_z\\
		V_w
	\end{bmatrix}.
\end{equation}
If the values of $V$, its first derivatives as well as the ``missing'' second
derivative (here $V_{zw}=\Delta V$) are prescribed at some generic point, the
above equations determine all higher derivatives of $V$ and hence $V$ itself.
The integrability conditions for the System~\eqref{eq:V''=CV'} are algebraic
equations in the coefficients $C_{ij}$ and their first derivatives, c.f.\ 
Equations~\eqref{eq:ICs}.  Now remember that the $C_{ij}$ were rational in the
$K^{(\alpha)}$ and their first derivatives.  Consequently, the integrability
conditions for the System~\eqref{eq:V''=CV'} are algebraic equations in the
$K^{(\alpha)}$ and their derivatives.  This also allows us to eliminate the unknown
function $V$.

Now observe that the (covariant) derivative is a linear operation.  The
integrability conditions of~\eqref{eq:V''=CV'} are therefore algebraic in the
$K^{(\alpha)}$ alone.  This defines a system of algebraic equations for $2n-1$
Killing tensors.  These equations determine whether $2n-1$ linearly
independent Killing tensors $K^{(\alpha)}$ can be completed to a solution of
the Equations~\eqref{eq:KidV=dVi}, necessary to define a superintegrable
system.

Our next observation is that a superintegrable system is a linear space,
expressed by the linearity of \eqref{eq:KidV=dVi} in
$(K^{(\alpha)},V^{(\alpha)})$.  This translates to the fact that the above
algebraic equations can be written in terms of the Plücker coordinates for the
vector space spanned by the $K^{(\alpha)}$.  Hence they define a subvariety in
the Grassmannian of ($2n-1$)-dimensional subspaces in the space of Killing
tensors.  It is this subvariety which we will call the \emph{variety of
superintegrable systems}.

In Section~\ref{sec:ASICs} we make the above explicit for superintegrable
systems in the Euclidean plane, in particular the algebraic equations defining
the variety of superintegrable systems.  For spaces where the classification
of superintegrable systems is known, it is a priori clear that these algebraic
equations can be solved explicitly.  However, the present example of the
Euclidean plane shows that it is much simpler to solve these equations from
scratch than to rewrite the known normal forms.  This is the content of
Section~\ref{sec:solution}.

In Section~\ref{sec:variety} we give a detailed description of the algebraic
geometric structure of the variety of superintegrable systems, such as the
irreducible components, their intersections, birational structure and
singularities.

% TODO
%
% \subsection{Applications}
%
% global study of the structure and representations of quadratic algebras,
% towards an algebraic geometric interpretation of the Askey scheme,

\bigskip
\noindent
\textbf{Acknowledgements.}
The authors would like to thank Andreas Vollmer and Vsevolod Shevchishin for
useful discussions on the subject.

%============================================================================%
\section{Preliminaries}
\label{sec:preliminaries}

%----------------------------------------------------------------------------%
\subsection{Superintegrable systems}

A Hamiltonian system is a dynamical system characterised by a Hamiltonian
function $H(\mathbf p,\mathbf q)$ on the phase space of positions $\mathbf
q=(q_1,\ldots,q_n)$ and momenta $\mathbf p=(p_1,\ldots,p_n)$.  Its temporal
evolution is governed by the equations of motion
\begin{align*}
	\dot{\mathbf p}&=-\frac{\partial H}{\partial\mathbf q}&
	\dot{\mathbf q}&=+\frac{\partial H}{\partial\mathbf p}
\end{align*}
A function $F(\mathbf p,\mathbf q)$ on the phase space is called a
\emph{constant of motion} or \emph{first integral}, if it is constant under
this evolution, i.e.\ if
\[
	\dot F
	=\frac{\partial F}{\partial\mathbf q}\dot{\mathbf q}
	+\frac{\partial F}{\partial\mathbf p}\dot{\mathbf p}
	=\frac{\partial F}{\partial\mathbf q}\frac{\partial H}{\partial\mathbf p}
	-\frac{\partial F}{\partial\mathbf p}\frac{\partial H}{\partial\mathbf q}
	=0
\]
or
\[
	\{F,H\}=0,
\]
where
\[
	\{F,G\}=
	\sum_{i=1}^n
	\left(
		\frac{\partial F}{\partial q_i}
		\frac{\partial G}{\partial p_i}
		-
		\frac{\partial G}{\partial q_i}
		\frac{\partial F}{\partial p_i}
	\right)
\]
is the canonical Poisson bracket.  Such a constant of motion restricts the
trajectory of the system to a hypersurface in phase space.  If the system
possesses the maximal number of $2n-1$ functionally independent constants of
motion $F^{(0)},\ldots,F^{(2n-2)}$, then its trajectory in phase space is the
(unparametrised) curve given as the intersection of the hypersurfaces
$F^{(\alpha)}(\mathbf p,\mathbf q)=c^{(\alpha)}$, where the constants
$c^{(\alpha)}$ are determined by the initial conditions.  For such systems we
can solve the equations of motion exactly and in a purely algebraic way,
without having to solve explicitly any differential equation.

\begin{definition}
	A \emph{maximally superintegrable system} is a Hamiltonian system
	admitting $2n-1$ functionally independent constants of motion
	$F^{(\alpha)}$,
	\begin{align}
		\label{eq:integral}
		\{F^{(\alpha)},H\}&=0&
		\alpha&=0,1,\ldots,2n-2,
	\end{align}
	one of which we can take to be the Hamiltonian itself:
	\[
		F^{(0)}=H.
	\]
	A superintegrable system is \emph{second order} if the constants of motion
	$F^{(\alpha)}$ are of the form
	\begin{equation}
		\label{eq:quadratic}
		F^{(\alpha)}=K^{(\alpha)}+V^{(\alpha)},
	\end{equation}
	where
	\[
		K^{(\alpha)}(\mathbf p,\mathbf q)=\sum_{i=1}^nK^{(\alpha)}_{ij}(\mathbf q)p^ip^j
	\]
	is quadratic in momenta and
	\[
		V^{(\alpha)}(\mathbf p,\mathbf q)=V^{(\alpha)}(\mathbf q)
	\]
	a potential function depending only on the positions.  In particular,
	\begin{equation}
		\label{eq:Hamiltonian}
		H=g+V,
	\end{equation}
	where
	\[
		g(\mathbf p,\mathbf q)=\sum_{i=1}^ng_{ij}(\mathbf q)p^ip^j
	\]
	is given by the Riemannian metric $g_{ij}(\mathbf q)$ on the underlying
	manifold.  We call $V$ a \emph{superintegrable potential} if the
	Hamiltonian \eqref{eq:Hamiltonian} defines a superintegrable system.
\end{definition}

In this article we will be concerned exclusively with second order maximally
superintegrable systems and thus omit the terms ``second order'' and
``maximally'' without further mentioning.

The condition \eqref{eq:integral} for \eqref{eq:quadratic} and
\eqref{eq:Hamiltonian} splits into two parts, which are cubic respectively
linear in $\mathbf p$:
\begin{subequations}
	\label{eq:1st+3rd}
	\begin{align}
		\label{eq:3rd}\{K^{(\alpha)},g\}&=0\\
		\label{eq:1st}\{K^{(\alpha)},V\}+\{V^{(\alpha)},g\}&=0
	\end{align}
\end{subequations}

%----------------------------------------------------------------------------%
\subsection{Killing tensors}

The condition $\{K,g\}=0$ for $K(\mathbf p,\mathbf q)=K_{ij}(\mathbf q)p^ip^j$
is equivalent to $K_{ij}$ being a Killing tensor in the following sense.
\begin{definition}
	A (second order) \emph{Killing tensor} is a symmetric tensor field on a
	Riemannian manifold satisfying the Killing equation
	\[
		K_{ij,k}+K_{jk,i}+K_{ki,j}=0,
	\]
	where the comma denotes covariant derivatives.
\end{definition}
Note that the metric $g$ is trivially a Killing tensor, since it is
covariantly constant.

%----------------------------------------------------------------------------%
\subsection{Bertrand-Darboux condition}

The metric $g$ also allows us to identify symmetric forms and endomorphisms.
Interpreting a Killing tensor in this way as an endomorphism on $1$-forms,
equation \eqref{eq:1st} can be written in the form
\[
	dV^{(\alpha)}=K^{(\alpha)}dV,
\]
and shows that, once the Killing tensors $K^{(\alpha)}$ are known, the
potentials $V^{(\alpha)}$ can be recovered from $V=V^{(0)}$ (up to an
irrelevant constant), provided the integrability condition
\begin{equation}
	\label{eq:dKdV}
	d(K^{(\alpha)}dV)=0
\end{equation}
is satisfied.  This eliminates the potentials $V^{(\alpha)}$ for
$\alpha\not=0$ from our equations.  In fact, as we will see below, the
remaining potential $V=V^{(0)}$ can be eliminated as well, leaving equations
on the Killing tensors $K^{(\alpha)}$ alone.

%----------------------------------------------------------------------------%
\subsection{Non-degeneracy}

For simplicity of notation, let us write $V_i$ and $V_{ij}$ for the first and
second derivatives of the scalar function $V$.  The Equations \eqref{eq:dKdV}
can be used to express the second derivatives $V_{ij}$, for $i\not=j$, and
$V_{ii}-V_{jj}$ as linear combinations of the first derivatives $V_i$, the
coefficients being rational expressions in $K^{(\alpha)}_{ij}$ and
$K^{(\alpha)}_{ij,k}$.  This determines all higher derivatives of $V$ at a
non-singular point if $V$, $\nabla V$ and $\Delta V$ are known at this point.
Therefore an analytic potential $V$ is uniquely determined by the Killing
tensors $K^{(\alpha)}$ if $V$, $\nabla V$ and $\Delta V$ are prescribed at a
non-singular point of $V$.  This motivates the following definition, following
\cite{Kalnins&Pogosyan&Miller}.

\begin{definition}
	A superintegrable system is called \emph{non-degenerate}, if $\Delta V$
	and the components of $\nabla V$ are linearly independent functions.
\end{definition}

This article is concerned with the classification of superintegrable systems
which are non-degenerate.

Note that, by the above, the Killing tensors $K^{(\alpha)}$ of a
non-degenerate superintegrable system uniquely define an $(n+2)$-dimensional
linear space of corresponding superintegrable potentials $V$, parametrised by
the values of $V$, $\nabla V$ and $\Delta V$ at a fixed non-singular point.

\begin{definition}
	\label{def:free+fibre}
	We call the $(2n-1)$-dimensional subspace in the space of Killing tensors
	that is spanned by the Killing tensors $K^{(\alpha)}$ of a
	(non-degenerate) superintegrable system the associated
	\emph{(non-degenerate) free superintegrable system}.  The
	$(n+2)$-dimensional space of superintegrable systems with the same
	associated free superintegrable system will be called the \emph{fibre}
	over this free superintegrable system.
\end{definition}

\begin{remark}
	\label{rem:free}
	Setting $V$ and $V^{(\alpha)}$ identically zero gives a solution of the
	Conditions~\eqref{eq:1st+3rd} for any choice of Killing tensors
	$K^{(\alpha)}$.  Thus a $(2n-1)$-dimensional subspace in the space of
	Killing tensors is a (not necessarily non-degenerate) free superintegrable
	system in the above sense if it contains the metric $g$.
\end{remark}

%----------------------------------------------------------------------------%
\subsection{Special conformal Killing tensors}

In dimension two Killing tensors can be described equivalently via special
conformal Killing tensors.  This will considerably simplify our
characterisation of superintegrable systems.

\begin{definition}
	A \emph{special conformal Killing tensor} is a symmetric tensor field
	$L_{ij}$ satisfying
	\begin{subequations}
		\label{eq:Sinjukov}
		\begin{equation}
			\label{eq:L}
			L_{ij,k}=\tfrac12(\lambda_ig_{jk}+\lambda_jg_{ik}),
		\end{equation}
		where $\lambda_i$ denotes the covariant derivative of
		\begin{equation}
			\label{eq:lambda}
			\lambda\coloneq\tr L,
		\end{equation}
	\end{subequations}
	as can be seen from contracting $i$ and $j$ in \eqref{eq:L}.
\end{definition}
Note that the metric $g$ is trivially a special conformal Killing tensor as
well, since it is covariantly constant.
\begin{lemma}
	\label{lem:iso}
	A special conformal Killing tensor $L$ determines a Killing tensor $K$ via
	\begin{equation}
		\label{eq:K(L)}
		K\coloneq L-(\tr L)g.
	\end{equation}
	On $2$-dimensional constant curvature manifolds this defines an
	isomorphism between the space of Killing tensors and the space of special
	conformal Killing tensors, mapping $g$ to $-g$.
\end{lemma}
\begin{proof}
	The Killing equation for $K$ is a direct consequence of
	\eqref{eq:Sinjukov}.  For the second part observe that the map defined by
	\eqref{eq:K(L)} is injective and that the dimension of both spaces is
	known to be six.
\end{proof}
We can rewrite the Bertrand-Darboux condition \eqref{eq:dKdV} for
\eqref{eq:K(L)} in terms of $L$ as
\[
	d(LdV)=d\lambda\wedge dV
\]
or, in local coordinates, as
\[
	\sum_{i=1}^n\Bigl(L\indices{^i_{[j}}V\indices{_{k]i}}+L\indices{^i_{[j,k]}}V_i\Bigr)=\lambda_{[k}V_{j]},
\]
where the square brackets denote antisymmetrisation in the enclosed indices.
Using \eqref{eq:Sinjukov} this becomes
\begin{equation}
	\label{eq:dLdV:ijk}
	\sum_{i=1}^nL\indices{^i_{[j}}V_{k]i}=\tfrac32\lambda_{[k}V_{j]}.
\end{equation}

%----------------------------------------------------------------------------%
\subsection{Complex Euclidean plane}

We will consider the Euclidean plane, i.e.\ a complex $2$-dimensional vector
space equipped with a complex scalar product~$g$.  The scalar product defines
a complex Riemannian metric which, by abuse of notation, will be denoted by
$g$ as well.  In the Euclidean plane every special conformal Killing tensor is
of the form
\begin{subequations}
	\label{eq:L:plane}
	\begin{align}
		L&=A+bx^T+xb^T+cxx^T
		\intertext{with trace}
		\label{eq:lambda:plane}
		\lambda&=\tr A+2b^Tx+cx^Tx
	\end{align}
\end{subequations}
where $A$, $b$ and $c$ are, respectively, a symmetric $2\times2$ matrix, a
vector and a scalar.  We will combine these parameters into a symmetric
$3\times3$ matrix
\begin{align}
	\label{eq:L:matrix}
	\hat L&\coloneq
	\begin{bmatrix}
		c&b^T\\
		b&A
	\end{bmatrix}&
	A^T&=A
\end{align}
parametrising the space of special conformal Killing tensors in the plane.  In
particular, as a special conformal Killing tensor the metric $g$ is given by
the symmetric matrix
\begin{equation}
	\label{eq:metric}
	\hat g=
	\begin{bmatrix}
		0&0\\
		0&g
	\end{bmatrix}.
\end{equation}
For convenience we will choose a null basis and corresponding coordinates $z$
and $w$:
\begin{align}
	\label{eq:gij}
	g_{zz}=g_{ww}&=0&
	g_{zw}=g_{wz}&=1.
\end{align}
The Bertrand-Darboux condition in the form \eqref{eq:dLdV:ijk} then reads
\begin{equation}
	\label{eq:dLdV:zw}
	L_{ww}V_{zz}-L_{zz}V_{ww}
	=\tfrac32(\lambda_zV_w-\lambda_wV_z).
\end{equation}

%----------------------------------------------------------------------------%
\subsection{Functional independence}

So far we have ignored the distinction between linear and functional
independence of the constants of motion.  Recall that functional independence
means that their differentials are linearly independent almost everywhere.
This condition can be formulated in a purely algebraic fashion and thus be
incorporated into our algebraic geometric description.  Being more pragmatic,
we would say we can check functional independence a posteriori.  However, it
turns out (or is known) that in our case functional and linear independence
are equivalent.  That is why we will ignore this distinction right from the
beginning.

%----------------------------------------------------------------------------%
\subsection{The variety of free superintegrable systems}

Let us denote by $S^2\C^3$ the space of complex symmetric $3\times 3$ matrices
and by $S^2_0\C^3$ the subspace of matrices of the form \eqref{eq:L:matrix}
with $\tr A=0$.  In a null basis we have $\tr A=2A_{zw}$, so that we can
represent elements in the five-dimensional space $S^2_0\C^3$ as vectors
\[
	\hat L=(A_{zz},2b_z,c,2b_w,A_{ww}).
\]
We have seen that the space of Killing tensors on the Euclidean plane is
naturally isomorphic to $S^2\C^3$.  Hence a free superintegrable system on the
Euclidean plane in the sense of Remark~\ref{rem:free} is given by a
three-dimensional subspace in $S^2\C^3$ containing \eqref{eq:metric} or,
equivalently, by a two-dimensional subspace in $S^2_0\C^3$.  Therefore the
free superintegrable systems on the Euclidean plane constitute a projective
variety isomorphic to the Grassmannian $G_2(S^2_0\C^3)$ of $2$-planes in
$S^2_0\C^3$, which can be embedded into $\P(\Lambda^2S^2_0\C^3)$ under the
Plücker embedding
\[
	\begin{array}{rcl}
		G_2(S^2_0\C^3)&\hookrightarrow&\P(\Lambda^2S^2_0\C^3)\\[\medskipamount]
		\operatorname{span}\bigl\{\hat L^{(1)},\hat L^{(2)}\bigr\}&\mapsto&\hat L^{(1)}\wedge\hat L^{(2)}.
	\end{array}
\]
For simplicity of notation let us use the simple superscripts ``$1$'' and
``$2$'' instead of ``$(1)$'' and ``$(2)$''.  Then the image of the above map
is the projective variety of rank two skew symmetric $5\times5$ matrices:
\begin{align}
	\label{eq:rank2skew}
	&\hat L^1\wedge\hat L^2\notag\\
	&=
	\renewcommand{\arraystretch}{1.25}
	\begin{bmatrix}
		0&2(A^1_{zz}b^2_z\!-\!b^1_zA^2_{zz})&A^1_{zz}c^2\!-\!c^1A^2_{zz}&2(A^1_{zz}b^2_w\!-\!b^1_wA^2_{zz})&A^1_{zz}A^2_{ww}\!-\!A^1_{ww}A^2_{zz}\\
		&0&2(b^1_zc^2\!-\!c^1b^2_z)&4(b^1_zb^2_w\!-\!b^1_wb^2_z)&2(b^1_zA^2_{ww}\!-\!A^1_{ww}b^2_z)\\
		&&0&2(c^1b^2_w\!-\!b^1_wc^2)&c^1A^2_{ww}\!-\!A^1_{ww}c^2\\
		&\text{(skew)}&&0&2(b^1_wA^2_{ww}\!-\!A^1_{ww}b^2_w)\\
		&&&&0\\
	\end{bmatrix}
	\notag\\
	&\eqcolon
	\renewcommand{\arraystretch}{1}
	\begin{bmatrix}
		0&a_{30}&a_{20}&a_{10}&a_{00}\\
		-a_{30}&0&a_{21}&a_{11}&a_{01}\\
		-a_{20}&-a_{21}&0&a_{12}&a_{02}\\
		-a_{10}&-a_{11}&-a_{12}&0&a_{03}\\
		-a_{00}&-a_{01}&-a_{02}&-a_{03}&0
	\end{bmatrix}
\end{align}
This variety is defined by the Plücker relations, given by the vanishing of
the Pfaffians of the five principal minors of the above matrix:
\begin{equation}
	\label{eq:Pluecker:aij}
	\begin{split}
		a_{03}a_{21}-a_{02}a_{11}+a_{01}a_{12}&=0\\
		a_{03}a_{20}-a_{02}a_{10}+a_{00}a_{12}&=0\\
		a_{03}a_{30}-a_{01}a_{10}+a_{00}a_{11}&=0\\
		a_{02}a_{30}-a_{01}a_{20}+a_{00}a_{21}&=0\\
		a_{12}a_{30}-a_{11}a_{20}+a_{10}a_{21}&=0.
	\end{split}
\end{equation}
These are the algebraic equations assuring that the skew symmetric
matrix~\eqref{eq:rank2skew} is of rank two (or zero).  Let us summarise the
above in the following Proposition.
\begin{proposition}
	\label{prop:configurationspace}
	The free superintegrable systems in the Euclidean plane constitute a
	projective variety isomorphic to the Grassmannian $G_2(S^2_0\C^3)$ of
	$2$-planes in $S^2_0\C^3$, embedded into $\P(\Lambda^2S^2_0\C^3)$ as the
	$6$-dimensional variety of rank two skew symmetric matrices
	\eqref{eq:rank2skew} given by the Equations~\eqref{eq:Pluecker:aij}.
\end{proposition}
In the next section we will show that free superintegrable systems which are
non-degenerate not only form a sub\emph{set}, but a sub\emph{variety} in the
above variety.  The non-degenerate superintegrable systems then form a fibre
bundle with $4$-dimensional linear fibres over this subvariety, explaining the
denomination in Definition~\ref{def:free+fibre}.

%============================================================================%
\section{The algebraic superintegrability conditions}
\label{sec:ASICs}

\subsection{Derivation}

We consider the Bertrand-Darboux conditions \eqref{eq:dKdV} for two Killing
tensors $K^{(1)}$ and $K^{(2)}$, rewritten in the form \eqref{eq:dLdV:zw} for
the corresponding special conformal Killing tensors $L^{(1)}$ and $L^{(2)}$:
\begin{equation}
	\label{eq:2dLdV}
	\begin{aligned}
		L^1_{ww}V_{zz}-L^1_{zz}V_{ww}&=\tfrac32(\lambda^1_zV_w-\lambda^1_wV_z)\\
		L^2_{ww}V_{zz}-L^2_{zz}V_{ww}&=\tfrac32(\lambda^2_zV_w-\lambda^2_wV_z).
	\end{aligned}
\end{equation}
As before, we write simple superscripts for simplicity of notation, as there
is no risk of confusion with exponents.  From \eqref{eq:lambda:plane} we
deduce that the second derivative of $\lambda$ is $\lambda_{ij}=cg_{ij}$.
Together with \eqref{eq:Sinjukov} and \eqref{eq:gij} we obtain for the
derivatives of the coefficients in \eqref{eq:2dLdV}:
\begin{align}
	\label{eq:L'}
	\renewcommand{\arraystretch}{1.25}
	\begin{aligned}
		L^i_{zz,z}&=0\qquad&L^i_{zz,w}&=\lambda^i_z\qquad&\lambda^i_{zz}&=0\qquad&\lambda^i_{wz}&=c^i\qquad\\
		L^i_{ww,w}&=0\qquad&L^i_{ww,z}&=\lambda^i_w\qquad&\lambda^i_{ww}&=0\qquad&\lambda^i_{zw}&=c^i\qquad
	\end{aligned}
\end{align}
for $i=1,2$.  Hence the derivatives of the coefficient determinant
\begin{subequations}
	\label{eq:D}
	\begin{equation}
		D\coloneq\det
		\begin{bmatrix}
			L^1_{zz}&L^2_{zz}\\
			L^1_{ww}&L^2_{ww}
		\end{bmatrix}
		=L^1_{zz}L^2_{ww}-L^1_{ww}L^2_{zz}
	\end{equation}
	are
	\begin{equation}
		\begin{aligned}
			D_z&=L^1_{zz}\lambda^2_w-\lambda^1_wL^2_{zz}&\qquad
			D_{zz}&=L^1_{zz}c^2-c^1L^2_{zz}&\qquad
			D_{zzz}&=0\\
			D_w&=\lambda^1_zL^2_{ww}-L^1_{ww}\lambda^2_z&
			D_{ww}&=c^1L^2_{ww}-L^1_{ww}c^2&
			D_{www}&=0
		\end{aligned}
	\end{equation}
	and
	\begin{equation}
		\begin{aligned}
			D_{zw}&=\lambda^1_z\lambda^2_w-\lambda^1_w\lambda^2_z&\qquad
			D_{zzw}&=\lambda^1_zc^2-c^1\lambda^2_z&\qquad
			D_{zzww}&=0\\
			&&
			D_{wwz}&=c^1\lambda^2_w-\lambda^1_wc^2.
		\end{aligned}
	\end{equation}
\end{subequations}
This shows that $D=D(z,w)$ is a cubic polynomial, but only quadratic in $z$
respectively in $w$.  In terms of the entries of the skew symmetric matrix
\eqref{eq:rank2skew}, this cubic is given by
\begin{equation}
	\label{eq:D:polynomial}
	D(z,w)
	=\sum a_{ij}z^iw^j,
\end{equation}
where the sum runs over all pairs $(i,j)\not=(2,2)$ with $0\le i,j\le2$.

Solving the linear system \eqref{eq:2dLdV} for $V_{zz}$ and $V_{ww}$, we
arrive at the system
\begin{equation}
	\label{eq:prolongation}
	\begin{bmatrix}
		V_{zz}\\
		V_{ww}
	\end{bmatrix}
	=
	\frac32
	\begin{bmatrix}
		C_{11}&C_{12}\\
		C_{21}&C_{22}
	\end{bmatrix}
	\begin{bmatrix}
		V_z\\
		V_w
	\end{bmatrix}
\end{equation}
with
\[
	\begin{bmatrix}
		C_{11}&C_{12}\\
		C_{21}&C_{22}
	\end{bmatrix}
	=\frac1{L^1_{zz}L^2_{ww}-L^1_{ww}L^2_{zz}}
	\begin{bmatrix}
		\lambda^1_wL^2_{zz}-L^1_{zz}\lambda^2_w&
		L^1_{zz}\lambda^2_z-\lambda^1_zL^2_{zz}\\
		\lambda^1_wL^2_{ww}-L^1_{ww}\lambda^2_w&
		L^1_{ww}\lambda^2_z-\lambda^1_zL^2_{ww}
	\end{bmatrix}.
\]
Comparing with \eqref{eq:D}, we see that the coefficient matrix can be written
as
\begin{equation}
	\label{eq:C}
	\begin{bmatrix}
		C_{11}&C_{12}\\
		C_{21}&C_{22}
	\end{bmatrix}
	=\frac1D
	\begin{bmatrix}
		-D_z&A_z\\
		B_w&-D_w
	\end{bmatrix}
\end{equation}
where
\begin{align*}
	A&\coloneq L^1_{zz}\lambda^2-\lambda^1L^2_{zz}&
	B&\coloneq\lambda^1L^2_{ww}-L^1_{ww}\lambda^2.
\end{align*}
By \eqref{eq:L'} the derivatives of $A_z$ and $B_w$ are
\begin{align}
	\label{eq:A'B'}
		A_{zz}&=0&
		A_{zw}&=D_{zz}&
		B_{ww}&=0&
		B_{wz}&=D_{ww}
\end{align}
and hence $A_z=A_z(w)$ and $B_w=B_w(z)$ are quadratic polynomials.  In terms
of the entries of the skew symmetric matrix \eqref{eq:rank2skew} they are
given by
\begin{align}
	\label{eq:AB}
	A_z(w)&=a_{21}w^2+2a_{20}w+a_{30}&
	B_w(z)&=a_{12}z^2+2a_{02}z+a_{03}.
\end{align}
That is, we can find the coefficients of $A_z$, $D$ and $B_w$ in,
respectively, the (overlapping) upper left, upper right and lower right
$3\times3$ submatrices of the skew symmetric $5\times 5$ matrix
\eqref{eq:rank2skew}.

From the expression \eqref{eq:C} together with \eqref{eq:A'B'} and the Plücker
relations \eqref{eq:Pluecker:aij} in the form \eqref{eq:Pluecker:AB} we easily
derive the following relations, where we have introduced two new symbols
$C_{122}=C_{12,w}$ and $C_{211}=C_{21,z}$:
\begin{align*}
	C_{11,w}&=C_{12}C_{21}&
	C_{12,z}&=C_{12}C_{11}&
	C_{11,z}&=C_{12}C_{22}+(C_{11})^2-C_{122}\\
	C_{22,z}&=C_{12}C_{21}&
	C_{21,w}&=C_{21}C_{22}&
	C_{22,w}&=C_{21}C_{11}+(C_{22})^2-C_{211}.
\end{align*}
The integrability conditions for these derivatives allow the derivatives of
$C_{122}$ and $C_{211}$ to be expressed as
\begin{align*}
  C_{122,z} &= C_{21}C_{122} + C_{22}C_{12}C_{21} \\
  C_{122,w} &= 2C_{11}C_{12}C_{21} + C_{12}(C_{22})^2 + C_{22}C_{122} - C_{12}C_{211} \\
  C_{211,z} &= 2C_{22}C_{12}C_{21} + C_{21}(C_{11})^2 + C_{11}C_{211} - C_{21}C_{112} \\
  C_{211,w} &= C_{12}C_{211} + C_{11}C_{12}C_{21}.
\end{align*}
We now have all derivatives of $C_{11}$, $C_{12}$, $C_{21}$, $C_{22}$, $C_{122}$ and $C_{211}$
expressed in terms of these symbols.  The remaining integrability conditions are generated by
\begin{subequations}
	\label{eq:ICs}
	\begin{equation}
		\begin{aligned}
			3C_{21}C_{122}-C_{11}C_{211}-C_{22}C_{12}C_{21}-C_{21}C_{11}C_{11}&=0\\
			3C_{12}C_{211}-C_{22}C_{122}-C_{11}C_{21}C_{12}-C_{12}C_{22}C_{22}&=0
		\end{aligned}
	\end{equation}
	and their differential consequence
	\begin{equation}
		\begin{split}
			2C_{122}C_{211}-C_{11}C_{21}C_{122}-C_{22}C_{12}C_{211}\\
			+C_{12}C_{12}C_{21}C_{21}-C_{11}C_{22}C_{12}C_{21}=0.
		\end{split}
	\end{equation}
\end{subequations}
Substituting \eqref{eq:C} into \eqref{eq:ICs} and taking \eqref{eq:A'B'} into
account, we obtain
\begin{subequations}
	\label{eq:SIC}
	\begin{equation}
		\label{eq:cubic}
		\begin{split}
			3A_zDD_{ww}-2A_zB_wD_z-2A_zD_w^2+DD_wD_{zz}&=0\\
			3B_wDD_{zz}-2A_zB_wD_w-2B_wD_z^2+DD_zD_{ww}&=0
		\end{split}
	\end{equation}
	as well as
	\begin{equation}
		\label{eq:quartic}
		2D^2D_{zz}D_{ww}-B_wDD_zD_{zz}-A_zDD_wD_{ww}-A_zB_wD_zD_w+A_z^2B_w^2=0.
	\end{equation}
\end{subequations}
Recall that $D$, $A_z$ and $B_w$ are polynomials in $z$ and $w$ with
coefficients $a_{ij}$.  The partial differential equations~\eqref{eq:SIC}
therefore define homogeneous algebraic equations in the $a_{ij}$.  In
combination with Proposition~\ref{prop:configurationspace} we now get our
first main result.

\begin{theorem}
	The set of non-degenerate (second order maximally) free superintegrable
	systems on the Euclidean plane has a natural structure of a projective
	variety, isomorphic to the subvariety in the variety of rank two skew
	symmetric $5\times5$ matrices \eqref{eq:rank2skew} defined by the
	algebraic equations \eqref{eq:SIC} for the polynomials
	\eqref{eq:D:polynomial} and \eqref{eq:AB}.
\end{theorem}

\begin{definition}
	For brevity, we will call the variety defined in the above proposition the
	\emph{variety of superintegrable systems}.
\end{definition}

\noindent
Some remarks concerning this definition are in order.

\begin{remark}
	We can regard the variety of superintegrable systems as a subvariety in
	the projective space of skew symmetric $5\times5$ matrices given by the
	homogeneous equations~\eqref{eq:SIC} together with the Plücker
	relations~\eqref{eq:Pluecker:aij}, the latter assuring that the matrix
	rank is two.
\end{remark}

\begin{remark}
	Degenerate superintegrable systems in dimension two turn out to be
	particular instances of non-degenerate systems
	\cite{Kalnins&Kress&MillerI,Kalnins&Kress&Miller&Post}.  This is why we
	omit ``non-degenerate'' from the name of the variety.
\end{remark}

\begin{remark}
	By construction, every non-degenerate free superintegrable system defines
	a point on the above variety.  However, there are two valid solutions of
	the superintegrability conditions~\eqref{eq:SIC} with $D$ vanishing
	identically.  Indeed, by \eqref{eq:A'B'} the polynomials $A_z$ and $B_w$
	are constant in this case and by \eqref{eq:Pluecker:aij} or
	\eqref{eq:quartic} one of them must vanish.
	Regarding~\eqref{eq:prolongation} and~\eqref{eq:C}, the variety of
	superintegrable systems therefore contains two points which do not
	correspond to a non-degenerate free superintegrable system.  We call them
	the two \emph{degenerate points} and denote their union by $\mathcal
	V_\varnothing$ for reasons to become clear later.
\end{remark}

\begin{remark}
	Working over the complex numbers allows us to treat both real cases at
	once:  In the Euclidean case we impose that $z$ and $w$ as well as
	$a_{ij}$ and $a_{ji}$ be complex conjugates and in the Minkowski case we
	impose them to be real.  The corresponding involutions
	$a_{ij}\mapsto\overline{a_{ji}}$ and $a_{ij}\mapsto\overline{a_{ij}}$
	define two real forms of the variety of superintegrable systems, which
	classify superintegrable systems on the real Euclidean plane respectively
	the Minkowski plane.

	The involution $a_{ij}\mapsto\overline{a_{ji}}$ is equivalent to
	exchanging $z$ and $w$ as well as $A$ and $B$.  We will refer to this
	operation as \emph{conjugation}.
\end{remark}

The variety of superintegrable systems captures the essential (difficult) part
of the classification problem, since all non-degenerate (free and non-free)
superintegrable systems form a fibre bundle with $4$-dimensional linear fibres
over this variety (excluding the two degenerate points).  Obtaining the fibre
over a point in the base amounts to a (simple) integration of the
System~\eqref{eq:prolongation} for given~$C_{ij}$, see
Section~\ref{sec:classification}.

%----------------------------------------------------------------------------%
\subsection{Simplification}

Writing
\[
	l^i=(L^i_{zz},\lambda^i_z,c,\lambda^i_w,L^i_{ww}),
	\qquad
	i=1,2,
\]
we can arrange the derivatives of $D$ together with $A_z$ and $B_w$ in
the rank two matrix
\[
	l^1\wedge l^2
	=
	\begin{bmatrix}
		0&A_z&D_{zz}&D_z&D\\
		&0&D_{zzw}&D_{zw}&D_w\\
		&&0&D_{wwz}&D_{ww}\\
		&\text{(skew)}&&0&B_w\\
		&&&&0\\
	\end{bmatrix}.
\]
Having rank two implies that the Pfaffians of its five principal minors
vanish.  For the $(3,3)$ minor this yields the identity
\begin{subequations}
	\label{eq:Pluecker}
	\begin{equation}
		\label{eq:Pluecker:AB}
		A_zB_w=D_zD_w-DD_{zw}.
	\end{equation}
	The remaining four principal minors are differential consequences of this
	identity, namely the two derivatives of this identity with respect to $z$
	respectively $w$:
	\begin{equation}
		\begin{aligned}
			\label{eq:Pluecker:AB'}
			A_zD_{ww}&=D_wD_{zz}-DD_{zzw}&\qquad
			A_zD_{wwz}&=D_{zz}D_{zw}-D_zD_{zzw}\\
			B_wD_{zz}&=D_zD_{ww}-DD_{wwz}&
			B_wD_{zzw}&=D_{zw}D_{ww}-D_wD_{wwz}
		\end{aligned}
	\end{equation}
\end{subequations}
All other derivatives are identically satisfied.  Note that this is nothing
but a local (differential) version of the Plücker relations
\eqref{eq:Pluecker:aij}, which are equivalent to \eqref{eq:Pluecker:AB}.

With the above identities we can transform the cubic superintegrability
condition \eqref{eq:cubic} into
\begin{equation}
	\label{eq:ABfromD}
	\begin{split}
		4DD_wD_{zz}-3D^2D_{zzw}-2D_wD_z^2+2DD_zD_{zw}-2A_zD_w^2&=0\\
		4DD_zD_{ww}-3D^2D_{wwz}-2D_zD_w^2+2DD_wD_{zw}-2B_wD_z^2&=0.
	\end{split}
\end{equation}
Differentiating the first condition with respect to $w$ and replacing the term
containing $A_z$ using the Plücker relations \eqref{eq:Pluecker:AB'} yields
\begin{multline}
	\label{eq:Donly}
	D(
		2D_{zz}D_{ww}+D_zD_{wwz}+D_wD_{zzw}+D_{zw}^2)\\
		-(D_z^2D_{ww}+D_w^2D_{zz}+D_zD_wD_{zw}
	)=0.
\end{multline}
Doing similarly with the second condition gives the same result.  In the same
way we can replace all terms containing $A_z$ or $B_w$ in the quartic
superintegrability condition \eqref{eq:quartic} using the Plücker relations
\eqref{eq:Pluecker}.  This yields Equation~\eqref{eq:Donly} multiplied by $D$.
Consequently, given the Plücker relations \eqref{eq:Pluecker}, we can 
confirm that the quartic
superintegrability condition \eqref{eq:quartic} is a differential consequence
of the cubic superintegrability conditions \eqref{eq:cubic}.  Rewriting
Equations~\eqref{eq:ABfromD} and~\eqref{eq:Donly}, we can summarise the above
as follows.

\begin{lemma}
	\label{lem:equations}
	The Plücker relations~\eqref{eq:Pluecker:aij} are equivalent
	to~\eqref{eq:Pluecker:AB}.  If they are satisfied, then the conditions
	\eqref{eq:SIC} are equivalent to
	\begin{subequations}
		\label{eq:AB3}
		\begin{align}
			\label{eq:A3} A_zD_w^2&=2DD_wD_{zz}-\tfrac32D^2D_{zzw}-D_wD_z^2+DD_zD_{zw}\\
			\label{eq:B3} B_wD_z^2&=2DD_zD_{ww}-\tfrac32D^2D_{wwz}-D_zD_w^2+DD_wD_{zw}
		\end{align}
	\end{subequations}
	and imply
	\begin{multline}
		\label{eq:D3}
		D_z^2D_{ww}+D_w^2D_{zz}+D_zD_wD_{zw}\\
		=D(2D_{zz}D_{ww}+D_zD_{wwz}+D_wD_{zzw}+D_{zw}^2).
	\end{multline}
\end{lemma}

Note that one can solve Equation~\eqref{eq:D3} for $D$ and then substitute the
solution into the Equations~\eqref{eq:AB3} or into the Plücker
relations~\eqref{eq:Pluecker} in order to determine $A_z$ and $B_w$.

\begin{corollary}
	\label{cor:equations}
	The variety of superintegrable systems is isomorphic to the subvariety in
	the projective space of skew symmetric $5\times5$
	matrices~\eqref{eq:rank2skew} given by the Plücker
	relations~\eqref{eq:Pluecker} and the Equations \eqref{eq:AB3} for the
	polynomials~\eqref{eq:D:polynomial} and~\eqref{eq:AB}.
\end{corollary}

Recall that $D$, $A_z$ and $B_w$ are polynomials in $z$ and $w$ with
coefficients $a_{ij}$.  The Equations~\eqref{eq:Pluecker} and~\eqref{eq:AB3}
therefore define homogeneous algebraic equations in the $a_{ij}$.  All
together, this gives a set of 5 quadratic and 32 cubic equations.  There will
be no need here to write them down explicitly.

\begin{definition}
	We call the defining equations of the variety of superintegrable systems
	the \emph{algebraic superintegrability conditions}.
\end{definition}

%============================================================================%
\section{Solution, normal forms and classification}
\label{sec:solution}

\subsection{Splitting of the ternary cubic}

The following lemma is the key observation for most of what follows.
\begin{lemma}
	\label{lem:decomposability}
	The algebraic superintegrability conditions imply that the ternary cubic
	$D(z,w)$ can be decomposed into linear factors.
\end{lemma}
\begin{proof}
	Differentiating Condition \eqref{eq:D3} with respect to $z$ respectively $w$
	results in
	\begin{subequations}
		\label{eq:D3:z,w}
		\begin{align}
			\label{eq:D3:z} D_wD_{zz}D_{zw}&=D(D_{zz}D_{wwz}+D_{zw}D_{zzw})\\
			\label{eq:D3:w} D_zD_{ww}D_{zw}&=D(D_{ww}D_{zzw}+D_{zw}D_{wwz}).
		\end{align}
	\end{subequations}
	Differentiating \eqref{eq:D3:z} with respect to $w$ or \eqref{eq:D3:w} with
	respect to $z$ we obtain
	\begin{equation}
		\label{eq:D3:zw}
		D_{zz}D_{zw}D_{ww}=2DD_{zzw}D_{wwz}.
	\end{equation}
	Recall that $D(z,w)$ is a cubic polynomial in $z$ and $w$.
	Hence the second derivatives $D_{zz}$, $D_{zw}$ and $D_{ww}$ are linear
	while the third derivatives $D_{zzw}$ and $D_{wwz}$ are constants.  We
	distinguish four cases, depending on whether these constants are zero or
	not.

	(1) If $D_{wwz}\not=0\not=D_{zzw}$, Equation~\eqref{eq:D3:zw}
	shows that $D$ decomposes into linear factors.

	(2) If $D_{wwz}=0\not=D_{zzw}$ we have $D_{zw}\not=0$
	and from \eqref{eq:D3:z} we deduce that $D$ is a constant multiple of
	$D_wD_{zz}$.  In particular we can assume $D_{zz}\not=0$ so that
	$D_{ww}=0$ by \eqref{eq:D3:zw}.  Consequently, $D_w$ only
	depends on $z$ and hence decomposes (over $\C$) into linear factors.  But
	then $D_wD_{zz}$ and therefore $D$ also decompose into linear factors.

	(3) The case $D_{zzw}=0\not=D_{wwz}$ becomes case (2) after
	interchanging $z$ and $w$.

	(4) In the remaining case $D_{zzw}=D_{wwz}=0$ the polynomial $D(z,w)$ is
	quadratic.  Hence the first derivatives $D_z$ and $D_w$ are linear and the
	second derivatives $D_{zz}$, $D_{zw}$ and $D_{ww}$ are constants.
	Equation~\eqref{eq:D3} then shows that $D$ is a constant multiple of
	$D_z^2D_{ww}+D_w^2D_{zz}+D_zD_wD_{zw}$.  If $D_{zw}\not=0$ then the
	Equations~\eqref{eq:D3:z,w} show that $D_zD_{ww}=D_wD_{zz}=0$.
	Therefore $D$ is a constant multiple of $D_zD_w$, which is a product of
	linear factors.  If $D_{zw}=0$ then $D$ is a constant multiple of
	$D_z^2D_{ww}+D_w^2D_{zz}$, which can also be decomposed (over $\C$) into
	linear factors.
\end{proof}
The above lemma allows us to write the cubic \eqref{eq:D} as a product of
three linear forms,
\begin{equation}
	\label{eq:D:splitting}
	D(z,w)=
	(a_1z+b_1w+c_1)
	(a_2z+b_2w+c_2)
	(a_3z+b_3w+c_3).
\end{equation}
The isometry group acts on each linear form as the dual of the standard
representation, c.f.\ Section~\ref{sec:normalforms}.  This action has three
orbits:  One orbit consisting of linear factors depending on $z$ alone (i.e.\ 
$b_i=0$), one consisting of linear factors depending on both $z$ and $w$ and
one consisting of factors depending on $w$ alone ($a_i=0$).  Accordingly we
denote the multiplicities of the linear factors in \eqref{eq:D:splitting} by a
triple with one label for each orbit (in this order).  The label ``$1$''
stands for a single factor, the label ``$2$'' for two proportional factors,
``$11$'' for two non-proportional factors and ``$0$'' for a constant.  Higher
multiplicities cannot appear due to the fact that $D(z,w)$ contains neither of
the cubic monomials $z^3$ and $w^3$.  See Table~\ref{tab:normalforms} for some
examples.

By~\eqref{eq:AB3}, the cubic $D(z,w)$ completely determines the
superintegrable system except for the degenerate cases where it does not
depend on $z$ or $w$.  That is why we will use these multiplicities to label
isometry classes of superintegrable systems.  Moreover, the factorisation
\eqref{eq:D:splitting} provides a geometric way to classify superintegrable
systems in the plane: Since each linear factor determines a projective line in
the plane, superintegrable systems can be labelled by planar arrangements of
three (possibly coinciding) projective lines, c.f.\ 
Figure~\ref{fig:inclusions}.  From \eqref{eq:prolongation} and \eqref{eq:C} we
see that these arrangements actually have an interpretation in terms of the
superintegrable potential:
\begin{proposition}
	The singular set of a superintegrable potential $V(z,w)$ in the fibre over
	a non-degenerate free superintegrable system is contained in a planar
	arrangement of (up to) three projective lines, given by the equation
	$D(z,w)=0$.
\end{proposition}
In summary, the splitting of the ternary cubic $D(z,w)$ provides an
intrinsic (algebraic as well as geometric) labelling scheme for
superintegrable systems.  In the next section we will see that it also allows
for a relatively simple solution of the algebraic superintegrability
conditions.

%----------------------------------------------------------------------------%
\subsection{Solution of the algebraic superintegrability conditions}

Owing to the fact that the cubic polynomial $D(z,w)$ is completely reducible
and only quadratic in $z$ and $w$, it must be of one of the following two
forms:
\begin{subequations}
	\begin{align}
		\label{eq:D:(1,1,1)}
		D(z,w)&=
		(a_1z+c_1)
		(a_2z+b_2w+c_2)
		(b_3w+c_3)\\
		\label{eq:D:(0,11,0)}
		D(z,w)&=
		(a_1z+b_1w+c_1)
		(a_3z+b_3w+c_3).
	\end{align}
\end{subequations}
In this parametrisation Equation~\eqref{eq:D3} for $D$ is easily solved.
\begin{proposition}
	\label{prop:down}
	As a set, the variety $\mathcal D$ of solutions to the
	Equation~\eqref{eq:D3} is the union
	\begin{equation}
		\label{eq:down:components}
		\mathcal D=
		\mathcal D_{( 1, 1, 1)}\cup
		\mathcal D_{(11, 0, 1)}\cup
		\mathcal D_{( 1, 0,11)}\cup
		\mathcal D_{( 0,11, 0)}
	\end{equation}
	of four classes, consisting of completely reducible ternary cubics of the
	following form.
	\begin{subequations}
		\label{eq:det}
		\begin{description}[wide,labelsep=1em,itemsep=1ex]
			\item[$\mathcal D_{(1,1,1)}$] $D(z,w)=(a_1z+c_1)(a_2z+b_2w+c_2)(b_3w+c_3)$
				subject to
				\begin{equation}
					\label{eq:det:abc}
					\det
					\begin{bmatrix}
						a_1&0&c_1\\
						a_2&b_2&c_2\\
						0&b_3&c_3
					\end{bmatrix}
					=0.
				\end{equation}
			\item[$\mathcal D_{(11,0,1)}$] $D(z,w)=(a_1z+c_1)(a_2z+c_2)(b_3w+c_3)$
			\item[$\mathcal D_{(1,0,11)}$] $D(z,w)=(a_1z+c_1)(b_2w+c_2)(b_3w+c_3)$
			\item[$\mathcal D_{(0,11,0)}$] $D(z,w)=(a_1z+b_1w+c_1)(a_3z+b_3w+c_3)$
				subject to\footnote{The minus sign in front of $b_3$ is correct.}
				\begin{equation}
					\label{eq:det:ab}
					\det
					\begin{bmatrix}
						a_1&b_1&c_1\\
						0&0&1\\
						a_3&-b_3&c_3
					\end{bmatrix}
					=a_1b_3+b_1a_3
					=0.
				\end{equation}
		\end{description}
	\end{subequations}
\end{proposition}
\begin{remark}
	The above classes are not disjoint.  Their intersections are given by
	Figure~\ref{fig:inclusions} with the following subclasses (and their
	conjugates):
	\begin{description}[wide,labelsep=1em,itemsep=1ex]
		\item[$\mathcal D_{(11,0,0)}$] $D(z,w)=(a_1z+c_1)(a_2z+c_2)$
		\item[$\mathcal D_{( 2,0,1)}$] $D(z,w)=(a_1z+c_1)^2(b_3w+c_3)$
		\item[$\mathcal D_{( 1,0,1)}$] $D(z,w)=(a_1z+c_1)(b_3w+c_3)$
	\end{description}
	These subclasses are not disjoint either.  Their intersections are given
	by the following subsubclasses (and their conjugates):
	\begin{description}[wide,labelsep=1em,itemsep=1ex]
		\item[$\mathcal D_{(2,0,0)}$] $D(z,w)=(a_1z+c_1)^2$
		\item[$\mathcal D_{(1,0,0)}$] $D(z,w)=(a_1z+c_1)$
		\item[$\mathcal D_{(0,1,0)}$] $D(z,w)=a_2z+b_2w+c_2$
	\end{description}
	Finally, these subsubclasses intersect in the class $\mathcal D_{(0,0,0)}$
	consisting of constant polynomials $D(z,w)$.
\end{remark}
\begin{figure}
	\centering
	\includegraphics[width=\textwidth]{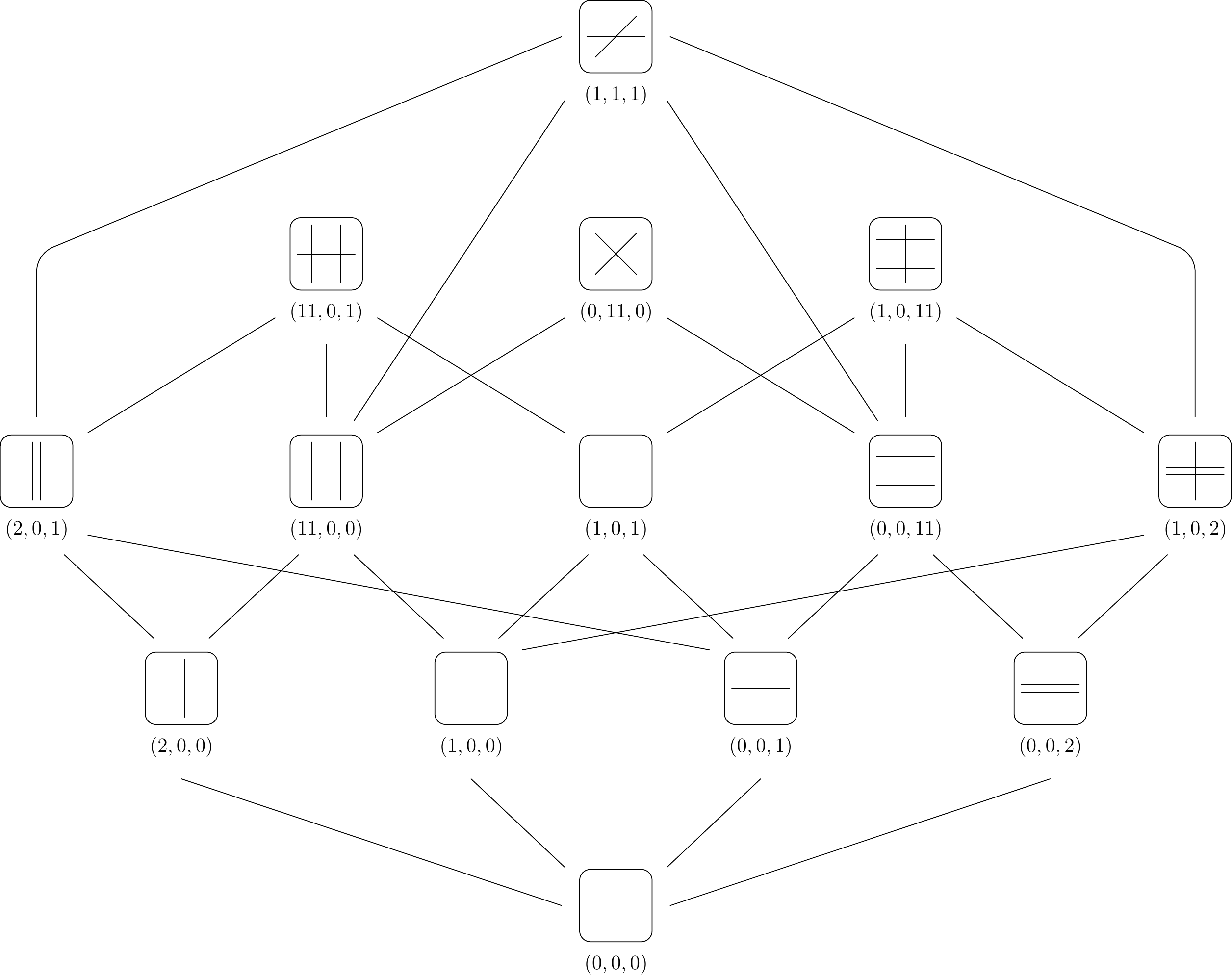}
	\caption{%
		Inclusion graph for classes of solutions of Equation~\eqref{eq:D3}.
		Inclusions are from bottom to top between classes joined by an edge.
		Two contiguous lines symbolise a double line.
	}
	\label{fig:inclusions}
\end{figure}
Corollary~\ref{cor:equations} now yields a complete solution of the algebraic
superintegrability conditions and hence a parametrisation of the variety of
superintegrable systems.
\begin{theorem}
	\label{thm:up:solution}
	As a set, the projective variety of superintegrable systems is the union%
	\footnote{%
		The ring accents indicate that the corresponding sets are not Zariski
		closed.
	}
	\begin{align*}
		\mathcal V=
		\mathring{\mathcal V}_{( 1, 1, 1)}&\cup
		\mathring{\mathcal V}_{(11, 0, 1)}\cup
		          \mathcal V _{(11, 0, 0)}\cup
		\mathring{\mathcal V}_{( 0,11, 0)}\\&\cup
		\mathring{\mathcal V}_{( 1, 0,11)}\cup
		          \mathcal V _{( 0, 0,11)}
	\end{align*}
	of the six classes given in Table~\ref{tab:solution} (up to conjugates).
\end{theorem}
\begin{proof}
	Recall that $D(z,w)$ determines the superintegrable system up to the two
	constants $a_{30}$ and $a_{03}$.  These can be determined from the
	superintegrability conditions~\eqref{eq:A3} and \eqref{eq:B3} or from the
	the Plücker relations~\eqref{eq:Pluecker}.  Solving the algebraic
	superintegrability conditions is therefore straightforward, so we will
	only justify the completeness of the list.

	Suppose $D$ lies in class $(1,1,1)$ with $a_1=0$ or $b_3=0$.  Without loss
	of generality we may suppose the latter.  In this case $c_3\not=0$,
	since otherwise $D$ would be identically zero, i.e.\ in class
	$(11,0,0)$.  Condition~\eqref{eq:det:abc} then implies $a_1=0$ or $b_2=0$,
	i.e.\ that $D$ lies in class $(0,1,0)$ or $(11,0,0)$.

	Note that if $D$ lies in class $(11,0,1)$ with $b_3=0$ then it lies in
	class $(11,0,0)$ and similarly for class $(1,0,11)$.

	Suppose now that $D$ is of class $(0,11,0)$ with $a_1a_3b_1b_3=0$.  Due to
	Condition~\eqref{eq:det:ab} we may assume without loss of generality that
	$b_1=a_1=0$ or $b_1=b_3=0$.  In the first case $D$ lies in class
	$(0,1,0)$, in the second case in class $(11,0,0)$.

	Suppose finally that $D$ lies in class $(0,1,0)$ with $a_2=0$ or $b_2=0$.
	In the first case $D$ also lies in class $(0,0,11)$ and in the second in
	class $(11,0,0)$.
\end{proof}

%----------------------------------------------------------------------------%
\subsection{Normal forms}
\label{sec:normalforms}

The set of superintegrable systems is invariant under isometries.  The variety
of superintegrable systems on the plane is therefore equipped with a natural
action of the Euclidean group.  On the polynomials $D$, $A$ and $B$ this
action is induced by the standard action of the Euclidean group on the plane.
In the null basis \eqref{eq:gij} translations and rotations are given by
shifts and shears, respectively, i.e.\ by
\begin{align*}
	(z,w)&\mapsto(z+c,w+d)&&c,d\in\C,\\
	(z,w)&\mapsto(\lambda z,w/\lambda)&&\lambda\in\C\setminus\{0\}.
\end{align*}
The different orbits and normal forms of this action can easily be derived
from Table~\ref{tab:solution} and are listed in Table~\ref{tab:normalforms}.

%----------------------------------------------------------------------------%
\subsection{Relative invariants}

In \cite{Kalnins&Kress&Miller} a complete set of relative invariants for
isometry classes of superintegrable systems was constructed.  We point out
that in our algebraic description this task is trivial:  The variables
$a_{ij}$ already constitute a complete set of relative invariants, as shown in
Table~\ref{tab:relativeinvariants}.

\begin{table}[h]
	\centering
	\begin{tabular}{cll}
		\toprule
		class & \multicolumn{1}{c}{vanishing relative invariants} & \multicolumn{1}{c}{relations} \\
		\midrule
		$( 1, 1, 1)$ & $\varnothing$ \\
		$(11, 0, 1)$ & $a_{12},a_{02},a_{03}$ \\
		$( 2, 0, 1)$ & $a_{12},a_{02},a_{03}$ & $a_{21}a_{01}=4a_{11}^2$, $a_{20}a_{00}=4a_{10}^2$ \\
		$( 0,11, 0)$ & $a_{12},a_{21},a_{11}$ \\
		$(11, 0, 0)$ & $a_{ij}$ unless $j=0$ \\
		$( 2, 0, 0)$ & $a_{ij}$ unless $j=0$ & $4a_{20}a_{00}=a_{10}^2$ \\
		$( 1, 0, 1)$ & $a_{ij}$ unless $0\le i,j\le1$ \\
		$( 0, 1, 0)$ & $a_{20},a_{21},a_{11},a_{12},a_{02}$ \\
		$( 1, 0, 0)$ & $a_{ij}$ unless $0\le i\le1$ and $j=0$ \\
		$( 0, 0, 0)$ & $a_{ij}$ unless $i=j=0$ \\
		\bottomrule
	\end{tabular}
	\bigskip
	\caption{%
		Relative invariants for isometry classes of superintegrable
		systems in the plane (up to conjugation).
	}
	\label{tab:relativeinvariants}
\end{table}

%----------------------------------------------------------------------------%
\subsection{Classification of superintegrable potentials}
\label{sec:classification}

By a separation of variables in $x=z+w$ and $y=z-w$ the system
\eqref{eq:prolongation} can be integrated for each normal form to yield the
corresponding superintegrable potentials as listed in
Table~\ref{tab:potentials}.  This perfectly matches the list
in~\cite{Kalnins&Kress&Pogosyan&Miller} and thereby confirms the present approach (or
the known classification).

%============================================================================%
\section{The variety of superintegrable systems}
\label{sec:variety}

After having solved the algebraic superintegrability conditions we now study
the geometric structure of the corresponding variety, i.e.\ the variety
$\mathcal V$ of superintegrable systems.  Recall that $\mathcal V$ is a
subvariety in the Grassmannian $G_2(S^2_0\C^3)$ of 2-planes in $S^2_0\C^3$,
embedded into $\P(\Lambda^2S^2_0\C^3)$ via the Plücker embedding, and that a
point on $\mathcal V$ is given by three polynomials $D$, $A_z$ and $B_w$.
Mapping $(D,A_z,B_w)\mapsto D$ defines a projection
\[
	\pi:\P(\Lambda^2S^2_0\C^3)\dashrightarrow\P(S^3\C^3)
\]
to the space of ternary cubics.  This map is defined on the complement of the
subspace $D=0$, which intersects $\mathcal V$ in the union $\mathcal
V_\varnothing$ of the two degenerate points.
% \[
% 	\begin{tikzcd}
% 		\pi^{-1}(\mathcal D_c)\arrow[r,hook]\arrow[d,shift right,dashed,"\pi     "']&\mathcal V\arrow[r,hook]\arrow[d,dashed]&\P(\Lambda^2S^2_0\C^3)\arrow[d,dashed,"\pi"]\\
% 		         \mathcal D_c \arrow[r,hook]\arrow[u,shift right,dashed,"\sigma_c"']&\Sigma_3\P^2\arrow[r,hook]&\P(S^3\C^3)
% 	\end{tikzcd}
% \]
\begin{center}
   \includegraphics[clip, trim=4.5cm 21.5cm 4.5cm 4.3cm, width=\textwidth]{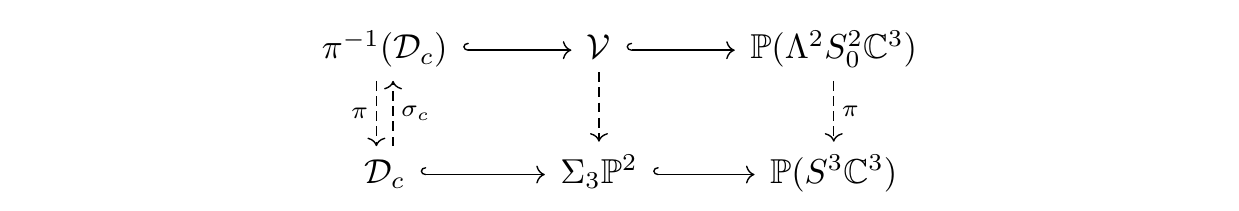}
\end{center}

By Lemma~\ref{lem:decomposability} the image $\mathcal D=\pi(\mathcal V)$
under this projection is contained in the subvariety of ternary cubics that
are decomposable into linear factors.  Denoting the symmetric product of three
projective planes by
\[
	\Sigma_3\P^2\coloneq(\P^2\times\P^2\times\P^2)/S_3,
\]
this subvariety is the image of the embedding of $\Sigma_3\P^2$ into the space
of ternary cubics, given by mapping the three linear factors to their product.

We now study the irreducible components~$\mathcal D_c\subset\mathcal D$ and
their preimages under~$\pi$.  The fact that generically the polynomials~$A_z$
and~$B_w$ are uniquely defined by~$D$ will provide rational right
inverses~$\sigma_c$ to the projection~$\pi$ over each irreducible
component~$\mathcal D_c$ and thereby a description of the irreducible
components in~$\mathcal V$.

\begin{proposition}
	Let $\hat{\mathcal D}_{(1,1,1)}$ be the variety of points
	\begin{subequations}
		\label{eq:resolution}
		\begin{equation}
			\label{eq:resolution:points}
			\bigl((a_1:c_1),(a_2:b_2:c_2),(b_3:c_3)\bigr)\in\P^1\times\P^2\times\P^1
		\end{equation}
		for which
		\begin{equation}
			\label{eq:resolution:det}
			\det
			\begin{bmatrix}
				a_1&0&c_1\\
				a_2&b_2&c_2\\
				0&b_3&c_3
			\end{bmatrix}
			=0.
		\end{equation}
	\end{subequations}
	Then the map given by sending \eqref{eq:resolution:points} to the cubic
	\eqref{eq:D:(1,1,1)} defines a regular birational map
	\[
		\hat{\mathcal D}_{(1,1,1)}\to\mathcal D_{(1,1,1)}.
	\]
	In particular, $\mathcal D_{(1,1,1)}$ is birational to
	$\P^1\times\P^1\times\P^1$ and hence irreducible.
\end{proposition}

\begin{proof}
	The above regular map is given explicitly by expanding
	\eqref{eq:D:(1,1,1)} and comparing it to \eqref{eq:D:polynomial}:
	\begin{align*}
		a_{21}&=a_1a_2b_3&a_{10}&=a_1c_2c_3+c_1a_2c_3&a_{11}&=a_1b_2c_3+a_1c_2b_3+c_1a_2b_3\\
		a_{12}&=a_1b_2b_3&a_{01}&=c_1b_2c_3+c_1c_2b_3\\
		a_{20}&=a_1a_2c_3\\
		a_{02}&=c_1b_2b_3\\
		a_{00}&=c_1c_2c_3.
	\end{align*}
	By \eqref{eq:resolution:det} we have $a_{11}=2a_1c_2b_3$, which gives the
	rational inverse
	\[
		(a_{ij})\mapsto\bigl((a_{12}:a_{02}),(a_{21}:a_{12}:\tfrac12a_{11}),(a_{21}:a_{20})\bigr).
	\]
	Finally, a birational isomorphism $\hat{\mathcal
	D}_{(1,1,1)}\subset\P^1\times\P^2\times\P^1\dashrightarrow\P^1\times\P^1\times\P^1$
	is given by the projection $(a_2:b_2:c_2)\mapsto(a_2:b_2)$ in the middle
	factor.
\end{proof}

Under the duality of points and lines in $\P^2$, $\hat{\mathcal D}_{(1,1,1)}$
is the variety of triples of collinear points, the first of them confined to
the line $w=0$, the third to $z=0$.  The second point is then confined to
the line between the other two unless they both coincide with the origin.

In view of Hironaka's Theorem, the following proposition shows that the map
$\hat{\mathcal D}_{(1,1,1)}\to\mathcal D_{(1,1,1)}$ above is ``almost'' a
resolution of the singularities in $\mathcal D_{(1,1,1)}$.

\begin{proposition}
	The variety $\hat{\mathcal D}_{(1,1,1)}$ is smooth on the complement of
	the point $\bigl((0:1),(0:0:1),(0:1)\bigr)$, which is the preimage of the
	point $\mathcal D_{(0,0,0)}$.
\end{proposition}

In the above dual picture, this singularity corresponds to the configuration
when all three points coincide with the origin.

\begin{proof}
	$\hat{\mathcal D}_{(1,1,1)}$ is the zero locus of the determinant map
	$\P^1\times\P^2\times\P^1\to\C$, given by sending
	\eqref{eq:resolution:points} to the left hand side of
	\eqref{eq:resolution:det}, and its singularities are those points where
	the tangent map vanishes.  The tangent of the determinant map
	$A\mapsto\det A$ is given by Jacobi's formula as $X\mapsto\tr XC^T$, where
	$C$ is the cofactor matrix of~$A$.  Hence $\hat{\mathcal D}_{(1,1,1)}$ is
	singular at points \eqref{eq:resolution:points} for which the cofactor
	matrix of the matrix in \eqref{eq:resolution:det} is orthogonal to the the
	space of matrices $X$ of the form
	\[
		\begin{bmatrix}
			*&0&*\\
			*&*&*\\
			0&*&*
		\end{bmatrix}
	\]
	with respect to the usual Hermitian inner product on matrices.  That is,
	the singular locus is given by the minors of the non-zero entries of the
	matrix in \eqref{eq:resolution:det}.  In particular, we have
	$a_1b_3=a_1c_3=0$.  Now $a_1\not=0$ would imply $b_3=c_3=0$, which is
	impossible.  So $(a_1:c_1)=(0:1)$ and similarly $(b_3:c_3)=(0:1)$.  We
	also have $c_1a_2=a_1c_2=0$ and $b_2c_3=c_2b_3=0$, from which we conclude
	$a_2=b_2=0$ since $c_1,c_3\not=0$.
\end{proof}

\begin{proposition}
	$\mathcal D_{(11,0,1)}$ is a variety biregular to $\P^2\times\P^1$.  So
	obviously, $\mathcal D_{(11,0,1)}$ is irreducible and smooth.  The same
	holds for $\mathcal D_{(1,0,11)}$, since it is conjugated to $\mathcal
	D_{(11,0,1)}$.
\end{proposition}

\begin{proof}
	This follows from the fact that $\mathcal D_{(11,0,1)}$ is the set of
	cubics of the form $(a_1z+c_1)(a_2z+c_2)(b_3w+c_3)$ and therefore
	biregular to the variety $(\Sigma_2\P^1)\times\P^1$, which is biregular
	to $\P^2\times\P^1$.
\end{proof}

Under the duality of points and lines in $\P^2$ we can regard $\mathcal
D_{(11,0,1)}$ as the variety of triples of unordered points with two of them
confined to the line $w=0$ and the third one to the line $z=0$.

\begin{proposition}
	$\mathcal D_{(0,11,0)}$ is a variety biregular to the cubic threefold of
	points
	\begin{subequations}
		\begin{equation}
			\label{eq:D_(0,11,0):points}
			(a_{20}:a_{02}:a_{10}:a_{01}:a_{00})\in\P^4
		\end{equation}
		for which
		\begin{equation}
			\label{eq:D_(0,11,0):matrix}
			\det
			\begin{bmatrix}
				a_{02}&a_{01}&0\\
				a_{01}&4a_{00}&a_{10}\\
				0&a_{10}&a_{20}
			\end{bmatrix}
			=0.
		\end{equation}
	\end{subequations}
	In particular, $\mathcal D_{(0,11,0)}$ is birational to $\P^3$ and hence
	irreducible.  Moreover, it is singular in the pair of intersecting lines
	$\mathcal D_{(2,0,0)}\cup\mathcal D_{(0,0,2)}$.
\end{proposition}

\begin{proof}
	$\mathcal D_{(0,11,0)}$ is the variety of ternary quadrics of the form
	\eqref{eq:D:(0,11,0)} subject to \eqref{eq:det:ab} and hence a subvariety
	in the variety $\Sigma_2\P^2$ of completely reducible ternary quadrics.
	Writing a ternary quadric as
	\[
		D(z,w)=a_{20}z^2+a_{02}w^2+a_{11}zw+a_{10}z+a_{01}w+a_{00},
	\]
	$\Sigma_2\P^2\subset\P(S^2\C^3)\cong\P^5$ is the hypersurface given by
	\[
		\det
		\begin{bmatrix}
			a_{02}&a_{01}&\frac12a_{11}\\
			a_{01}&4a_{00}&a_{10}\\
			\frac12a_{11}&a_{10}&a_{20}
		\end{bmatrix}
		=0.
	\]
	Expanding \eqref{eq:D:(0,11,0)}, comparing it to \eqref{eq:D:polynomial}
	and taking \eqref{eq:det:ab} into account, the subvariety $\mathcal
	D_{(0,11,0)}\subset\Sigma_2\P^2$ is seen to be the linear section
	$a_{11}=0$.

	A birational isomorphism $\mathcal
	D_{(0,11,0)}\subset\P^4\dashrightarrow\P^3$ is given, for example, by
	projecting \eqref{eq:D_(0,11,0):points} onto the first four homogeneous
	coordinates.

	The singular locus of $\mathcal D_{(0,11,0)}$ can be computed similarly to
	that of $\mathcal D_{(1,1,1)}$ above.  It is given by the minors of the
	non-zero entries of the matrix in \eqref{eq:D_(0,11,0):matrix}, which
	define the subvariety $\mathcal D_{(2,0,0)}\cup\mathcal D_{(0,0,2)}$.
\end{proof}

Since all four components are irreducible and cover $\mathcal D$, we have the
following.

\begin{corollary}
	The decomposition~\eqref{eq:down:components} of $\mathcal D=\pi(\mathcal
	V)$ is a decomposition into irreducible components.
\end{corollary}

We finally state our last main result, the structure theorem for the variety
of superintegrable systems in the Euclidean plane.  It shows that the
non-trivial components of $\mathcal V$ are blowups of the components of
$\mathcal D$ in certain pairs of intersecting lines.

\begin{theorem}
	The variety of superintegrable systems has a decomposition into six
	irreducible components,
	\begin{align*}
		\mathcal V=
		\mathcal V_{( 1, 1, 1)}&\cup
		\mathcal V_{(11, 0, 1)}\cup
		\mathcal V_{(11, 0, 0)}\cup
		\mathcal V_{( 0,11, 0)}\\&\cup
		\mathcal V_{( 1, 0,11)}\cup
		\mathcal V_{( 0, 0,11)},
	\end{align*}
	which are the Zariski closures of the corresponding classes given in
	Table~\ref{tab:solution}.  The projection $\pi:\mathcal
	V\dashrightarrow\mathcal D$ restricts to regular maps
	\begin{align*}
		(i)  \;\;&\mathcal V_{(1 ,1,1)}\rightarrow\mathcal D_{(1 ,1,1)}&
		(ii) \;\;&\mathcal V_{(11,0,1)}\rightarrow\mathcal D_{(11,0,1)}&
		(iv) \;\;&\mathcal V_{(0,11,0)}\rightarrow\mathcal D_{(0,11,0)}\\&&
		(iii)\;\;&\mathcal V_{(1,0,11)}\rightarrow\mathcal D_{(1,0,11)},
	\end{align*}
	each of which is an isomorphism over the complement of a pair of
	intersecting lines, namely:
	\begin{align*}
		(i-iii)\;\;&\mathcal D_{(1,0,0)}\cup\mathcal D_{(0,0,1)}&
		(iv)   \;\;&\mathcal D_{(2,0,0)}\cup\mathcal D_{(0,0,2)}
	\end{align*}
	On the remaining two components, the projection $\pi:\mathcal
	V\dashrightarrow\mathcal D$ restricts to central projections
	\begin{align*}
		(v)  \;\;&\mathcal V_{(11,0,0)}\dashrightarrow\mathcal D_{(11,0,0)}\\
		(vi) \;\;&\mathcal V_{(0,0,11)}\dashrightarrow\mathcal D_{(0,0,11)},
	\end{align*}
	each from one of the two degenerate points.
\end{theorem}

\begin{proof}
	The projection $\pi$ is regular on the complement of the line where $D$ is
	identically zero.  This line intersects $\mathcal V$ exactly in the two
	degenerate points, which are contained in $\mathcal V_{(11,0,0)}$
	respectively $\mathcal V_{(0,0,11)}$, but not in the other components.  To
	define the required birational inverses, recall that the projection map
	$\pi$ ``forgets'' the two coefficients $a_{30}$ and $a_{03}$, so that we
	have to recover them from the remaining $a_{ij}$.  From the Plücker
	relations~\eqref{eq:Pluecker:aij} we get
	\[
		a_{30}
		=\frac{a_{20}a_{11}-a_{21}a_{10}}{a_{12}}
		=\frac{a_{20}a_{01}-a_{21}a_{00}}{a_{02}}
	\]
	and from evaluating \eqref{eq:A3} at $z=w=0$ we obtain
	\[
		a_{30}=\frac{a_{00}(4a_{10}a_{02}-3a_{00}a_{12})+a_{01}(a_{00}a_{11}-a_{10}a_{01})}{a_{01}^2}.
	\]
	On the other hand, from the explicit solution in
	Table~\ref{tab:solution} we see that
	\begin{align*}
		\text{on $\mathcal V_{( 1, 1,1)}$:}\quad a_{30}&=\frac{a_{20}}{a_{21}}a_{20}\\
		\text{on $\mathcal V_{(11, 0,1)}$:}\quad a_{30}&=\frac{a_{20}}{a_{21}}a_{20}=\frac{a_{10}}{a_{11}}a_{20}=\frac{a_{00}}{a_{01}}a_{20}\\
		\text{on $\mathcal V_{( 0,11,0)}$:}\quad a_{30}&=\frac{a_{20}}{a_{02}}a_{01}=\frac{4a_{00}a_{20}-a_{10}^2}{a_{01}}.
	\end{align*}
	The coefficient $a_{30}$ is therefore well defined on the complement of
	the common zero locus of all nominators and denominators in the above
	quotients.  On $\mathcal D_{(11,0,1)}$ and $\mathcal D_{(0,11,0)}$ this is
	readily seen to be $\mathcal D_{(1,0,0)}$ respectively $\mathcal
	D_{(2,0,0)}$.  On $\mathcal D_{(1,1,1)}$ the square of the left hand side
	of \eqref{eq:resolution:det} can be expressed as
	\[
		a_{11}^2+8a_{20}a_{02}-4(a_{01}a_{21}+a_{10}a_{12})=0
	\]
	and implies that also $a_{11}=0$.  Therefore $a_{30}$ is well defined on
	the complement of $\mathcal D_{(1,0,0)}$ in $\mathcal D_{(1,1,1)}$.  On
	$\mathcal V_{(1,0,11)}$ we have $a_{30}=0$, which is well defined anyway.
	Similar statements hold for $a_{03}$ by interchanging $a_{ji}$ and
	$a_{ij}$.  This gives a rational left inverse to the projection over each
	component which is regular on the claimed complements.  In this way we get
	four of the six components.

	For the last statement, note that $\mathcal D_{(11,0,0)}$ is the variety
	of quadrics $D(z,w)=a_{20}z^2+a_{10}z+a_{00}$ and that $\mathcal
	V_{(11,0,0)}$ is the subvariety given by $a_{ij}=0$ for $j\not=0$.  This
	accounts for the remaining two components.
\end{proof}

%============================================================================%
\appendix
\section{Tables}

\begin{table}
	\renewcommand{\arraystretch}{1.25}
	\begin{tabular}{lllll}
		\toprule
		class & polynomials & conditions \\
		\midrule
		$\mathring{\mathcal V}_{(1,1,1)}$
			& $D(z,w)=(a_1z+c_1)(a_2z+b_2w+c_2)(b_3w+c_3)$ & \eqref{eq:det:abc} \\
			& $A_z(w)=a_1a_2(b_3w+c_3)^2/b_3$ & $b_3\not=0$ \\ \multicolumn{1}{c}{$\cup$}
			& $B_w(z)=b_2b_3(a_1z+c_1)^2/a_1$ & $a_1\not=0$ \\[1ex] \cline{2-3} \\[-2.5ex]
		$\mathring{\mathcal V}_{(0,1,0)}$
			& $D(z,w)=a_2z+b_2w +c_2$ \\
			& $A_z(w)=-a_2^2/b_2$ & $b_2\not=0$ \\
			& $B_w(z)=-b_2^2/a_2$ & $a_2\not=0$ \\
		\midrule
		$\mathring{\mathcal V}_{(11,0,1)}$
			& $D(z,w)=(a_1z+c_1)(a_2z+c_2)(b_3w+c_3)$ \\
			& $A_z(w)=a_1a_2(b_3w+c_3)^2/b_3$ & $b_3\not=0$ \\
			& $B_w(z)=0$ \\
%		$\mathring{\mathcal V}_{(1,0,11)}$
%			& $D(z,w)=(a_1z+c_1)(b_2w +c_2)(b_3w+c_3)$ & $a_1\not=0$ \\
%			& $A_z(w)=0$ \\
%			& $B_w(z)=b_2b_3(a_1z+c_1)^2/a_1$ \\
		\midrule
		$\mathring{\mathcal V}_{(0,11,0)}$
			& $D(z,w)=(a_1z+b_1w+c_1)(a_3z+b_3w+c_3)$ & \eqref{eq:det:ab} \\
			& $A_z(w)=a_1a_3\bigl(2w+c_3/b_3+c_1/b_1\bigr)$ & $b_1b_3\not=0$ \\
			& $B_w(z)=b_1b_3\bigl(2z+c_3/a_3+c_1/a_1\bigr)$ & $a_1a_3\not=0$ \\
		\midrule
		$\mathcal V_{(11,0,0)}$
			& $D(z,w)=(a_1z+c_1)(a_2z+c_2)$ \\
			& $A_z(w)=2a_1a_2w+a_{30}$ & none \\
			& $B_w(z)=0$ \\
%		$\mathcal V_{(0,0,11)}$
%			& $D(z,w)=(b_2w +c_2)(b_3w+c_3)$ & none \\
%			& $A_z(w)=0$ \\
%			& $B_w(z)=2b_2b_3z+a_{03}$ \\
		\bottomrule
	\end{tabular}
	\bigskip
	\caption{%
		Complete solution of the algebraic superintegrability conditions.  The
		ring accents indicate that the corresponding sets are not Zariski
		closed.  We define the class $\mathring{\mathcal V}_{(1,1,1)}$ to
		comprise the class $\mathring{\mathcal V}_{(0,1,0)}$, since it is a
		limiting case.
	}
	\label{tab:solution}
\end{table}

\begin{table}
	\begin{tabular}{cllll}
		\toprule
		class & \multicolumn{1}{c}{$D(z,w)$} & \multicolumn{1}{c}{$A_z(w)$} & \multicolumn{1}{c}{$B_w(z)$} & label \\
		\midrule
		$( 1, 1, 1)$ & $z(z+w)w$ & $w^2$ & $z^2$ & E16 \\
		$(11, 0, 1)$ & $z(z+1)w$ & $w^2$ & $0$ & E19 \\
		$( 1, 0,11)$ & $z(w+1)w$ & $0$ & $z^2$ & E19 \\
		$( 2, 0, 1)$ & $z^2w$ & $w^2$ & $0$ & E17 \\
		$( 1, 0, 2)$ & $zw^2$ & $0$ & $z^2$ & E17 \\
		\midrule
		$( 0,11, 0)$ & $(z+w)(z-w)$ & $2w$ & $-2z$ & E1 \\
		$(11, 0, 0)$ & $z(z+1)$ & $2w$ & $0$ & E7 \\
		$( 0, 0,11)$ & $w(w+1)$ & $0$ & $2z$ & E7 \\
		$( 2, 0, 0)$ & $z^2$ & $2w$ & $0$ & E8 \\
		$( 0, 0, 2)$ & $w^2$ & $0$ & $2z$ & E8 \\
		$( 1, 0, 1)$ & $zw$ & $0$ & $0$ & E20 \\
		\midrule
		$( 0, 1, 0)$ & $z+w$ & $1$ & $1$ & E2 \\
		$( 1, 0, 0)$ & $z$ & $1$ & $0$ & E9 \\
		$( 1, 0, 0)$ & $z$ & $0$ & $0$ & E11 \\
		$( 0, 0, 1)$ & $w$ & $0$ & $1$ & E9 \\
		$( 0, 0, 1)$ & $w$ & $0$ & $0$ & E11 \\
		\midrule
		$( 0, 0, 0)$ & $1$ & $1$ & $0$ & E10 \\
		$( 0, 0, 0)$ & $1$ & $0$ & $1$ & E10 \\
		$( 0, 0, 0)$ & $1$ & $0$ & $0$ & E3 \\
		\bottomrule
	\end{tabular}
	\bigskip
	\caption{%
		Normal forms for solutions of the algebraic superintegrability
		conditions with corresponding labels from
		\cite{Kalnins&Kress&Pogosyan&Miller}.
	}
	\label{tab:normalforms}
\end{table}

\begin{table}
	\centering
	\renewcommand{\arraystretch}{1.5}
	\begin{tabular}{cccc}
		\toprule
		class & \multicolumn{3}{c}{superintegrable potentials} \\
		\midrule
		$( 1, 1, 1)$ & $\frac1{\sqrt{zw}}$ & $\frac1{\sqrt{zw}}\frac1{(\sqrt z+\sqrt w)^2}$ & $\frac1{\sqrt{zw}}\frac1{(\sqrt z-\sqrt w)^2}$ \\
		$(11, 0, 1)$ & $\frac w{\sqrt{(w+1)(w-1)}}$ & $\frac1{\sqrt{z(w+1)}}$ & $\frac1{\sqrt{z(w-1)}}$ \\
		$( 2, 0, 1)$ & $\frac1{\sqrt{zw}}$ & $\frac1{w\sqrt{zw}}$ & $\frac1{w^2}$ \\
		$( 0,11, 0)$ & $zw$ & $\frac1{x^2}$ & $\frac1{y^2}$ \\
		$(11, 0, 0)$ & $zw$ & $\frac w{\sqrt{w^2-1}}$ & $\frac{2zw^2-z}{\sqrt{w^2-1}}$ \\
		$( 2, 0, 0)$ & $zw$ & $\frac1{w^2}$ & $\frac z{w^3}$ \\
		$( 1, 0, 1)$ & $\frac1{\sqrt{zw}}$ & $\frac1{\sqrt z}$ & $\frac1{\sqrt w}$ \\
		$( 0, 1, 0)$ & $x^2+4y^2$ & $\frac1{x^2}$ & $y$ \\
		$( 1, 0, 0)$ & $\frac z{\sqrt w}$ & $\frac1{\sqrt w}$ & $z$ \\
		$( 1, 0, 0)$ & $\frac1{\sqrt w}$ & $x$ & $\frac{z+3w}{\sqrt w}$ \\
		$( 0, 0, 0)$ & $zw$ & $z$ & $w$ \\
		$( 0, 0, 0)$ & $w^3+3zw$ & $w^2+z$ & $w$ \\
		\bottomrule
	\end{tabular}
	\bigskip
	\caption{%
		Superintegrable potentials in the plane (up to conjugation).  For each
		class a basis of non-constant superintegrable potentials is given.
	}
	\label{tab:potentials}
\end{table}

%============================================================================%
\clearpage
\bibliographystyle{amsalpha}
\bibliography{E2}
\medskip

\end{document}